\documentclass[preprint,12pt]{elsarticle}
\usepackage{amsfonts,amssymb,amsmath,mathrsfs}
\usepackage{lineno,hyperref}
\modulolinenumbers[5]

\usepackage{color}

\journal{...}

\newtheorem{theorem}{Theorem}[section]
\newtheorem{lemma}[theorem]{Lemma}
\newtheorem{example}[theorem]{Example}

\newtheorem{proposition}[theorem]{Proposition}
\newtheorem{remark}[theorem]{Remark}
\newtheorem{corollary}[theorem]{Corollary}
\newtheorem{definition}[theorem]{Definition}
\numberwithin{equation}{section}

\def\T{\mathbb{T}}

\def\R{\mathbb{R}}
\def\Z{\mathbb{Z}}
\def\N{\mathbb{N}}
\def\C{\mathbb{C}}
\def\Q{\mathbb{Q}}
\def\D{\mathcal{D}'}

\newenvironment{proof}[1][\noindent \textbf{Proof: }]{#1}{ \hfill $\square$ \vspace{2mm}}

\begin{document}

\begin{frontmatter}
	
	\title{Time-periodic Gelfand-Shilov spaces and global hypoellipticity on $\T\times \R^n$}


	\author[addressUFPR]{Fernando de \'Avila Silva}
	\ead{fernando.avila@ufpr.br}


\author[addressUNITO]
	{Marco Cappiello\corref{correspondingauthor}}
	\cortext[correspondingauthor]{Corresponding author}
	\ead{marco.cappiello@unito.it}

	\address[addressUFPR]{Department of Mathematics, Federal University of Paran\'a, Caixa Postal 19081, \\ CEP 81531-980, Curitiba, Brazil}

	\address[addressUNITO]{Department of Mathematics, University of Turin, Via Carlo Alberto 10, 10123, Turin, Italy}

	\begin{abstract}
We introduce a class of time-periodic Gelfand-Shilov spaces  of  functions on $\T \times \R^n$, where $\T \sim\mathbb{\R} /2\pi \Z$ is the one-dimensional torus. We develop a Fourier analysis inspired by the characterization of the Gelfand-Shilov spaces in terms of the eigenfunction expansions given by a fixed normal, globally elliptic differential operator on $\R^n$. In this setting, as an application, we characterize the global hypoellipticity for a class of linear differential evolution operators on $\T \times \R^n$.
	\end{abstract}
	
	\begin{keyword}
		Gelfand-Shilov spaces, periodic equations, global hypoellipticity, Fourier analysis
		\MSC[2020] Primary 46F05, 35B10, 35B65, 35H10
	\end{keyword}

\end{frontmatter}
\vskip0.5cm
This paper is dedicated to our friend and colleague Todor V. Gramchev.

\section{Introduction}
Gelfand-Shilov spaces of type $\mathscr{S}$ have been first introduced in \cite{GS2} as an alternative functional  setting to the Schwartz space $\mathscr{S}(\R^n)$ of smooth rapidly decreasing functions for the study of partial differential equations, cf. \cite{GS3}. Namely, fixed $ \mu>0, \nu >0$ such that $\mu+\nu \geq 1$, the space $\mathcal{S}^\mu_\nu(\R^n)$ is defined as the space of all functions $f \in C^\infty(\R^n)$ satisfying the following estimate
\begin{equation}\label{GSdef}
	\sup_{\alpha, \beta \in \N^n} \sup_{x \in \R^n} A^{-|\alpha+\beta|} \alpha!^{-\nu}\beta!^{-\mu}
	|x^\alpha \partial_x^\beta f(x)| <+\infty \end{equation}
for some $A>0$, or equivalently,
\begin{equation}\label{GSdef2}
	\sup_{\beta \in \N^n} \sup_{x \in \R^n} C^{-|\beta|} \beta!^{-\mu}\exp( c|x|^{1/\nu})
	| \partial_x^\beta f(x)| <+\infty \end{equation}
for some $C,c>0$.
Elements of $\mathcal{S}^\mu_\nu(\R^n)$ are then smooth functions presenting uniform analytic or Gevrey estimates on $\R^n$ and admitting an exponential decay at infinity. The elements of the dual space $(\mathcal{S}^\mu_\nu)'(\R^n)$ are commonly known as \textit{temperate ultradistributions}, cf. \cite{Pil}. Functional properties of these spaces have been studied along the years by several authors, cf. \cite{Av, CCK, MIT, Pil}.  
In the last twenty years, these spaces and their generalizations have been rediscovered as a suitable functional setting for Fourier, microlocal and time-frequency analysis, see \cite{ACT, AsensioLopez, BJO,  CapRend,  CGR3, CR, CT, GZ, Prangoski}. In particular, several classes of pseudodifferential operators have been studied in the Gelfand-Shilov setting and this led to a big number of applications and results concerning elliptic and evolution partial differential equations, see \cite{ACJDE, ACJMPA, CGR1, CGR2, CGR3, CaWa, CoNiRo1, CoNiRo2}. Some of these results concern global hypoellipticity for Shubin and SG pseudodifferential operators, see \cite{CGR1, CGR2, CGR3}. Recently, these spaces have been also characterized in terms of eigenfunction expansions, see \cite{CGPR, GPR}. 

The main goal of the paper is to introduce a new class of \textit{time-periodic} Gelfand-Shilov spaces and to extend to them the characterization given in \cite{GPR} for classical Gelfand-Shilov spaces. We emphasize that many  problems in analysis are connected with  periodic differential equations, as the models in biology, physics,  and engineering. Furthermore, a number of these problems consider classes of generalized functions and distributions.  In particular, a great advantage in the periodic analysis is a possible discretization of the involved equations in terms of Fourier series. In the present paper, the motivation and the application of Fourier expansions will concern global hypoellipticity for a class of linear evolution operators. We shall restrict for simplicity to the symmetric case $\mu=\nu$. Before treating the general case we need to find a characterization of the non-symmetric Gelfand-Shilov spaces via eigenfunction expansions which at this moment exists only in particular cases, cf. \cite{CGPR} for details.

For an outline of our main results and techniques, let 
$\T \sim\mathbb{\R} /2\pi \Z$ be the one-dimensional torus and consider $\sigma \geq 1$ and $\mu\geq 1/2$ be fixed. We are interested in the study of smooth, complex-valued functions $u$ defined on $\T \times \R^n$ satisfying the following: there exist $R,C >0$ such that
\begin{equation}\label{def-periodic-gelf-intro}
	\sup_{t \in \T, x \in \R^n} |x^\alpha \partial_x^\beta \partial_t^k u(t,x)| \leq RC^{|\alpha+\beta|+k}(k!)^\sigma (\alpha!\beta!)^\mu
\end{equation}
for every $\alpha, \beta \in \N^n_0, k \in \N_0$.  Broadly speaking, this is the class of functions that  belong to the symmetric Gelfand-Shilov spaces $\mathcal{S}_{\mu}^{\mu} (\R^n)$ with respect to the variable $x$, while are Gevrey regular and periodic in $t$.

We aim to characterize the spaces of functions satisfying \eqref{def-periodic-gelf-intro} by using  a discretization approach based on the Fourier expansions on  $\R^n$ introduced in \cite{GPR}. (See also \cite{CGPR}). Precisely, given an elliptic operator $P(x,D_x)$ on $\R^n$, satisfying suitable assumptions,  we use its orthonormal basis of eigenfunctions $\{\varphi_j(x)\}_{j \in \N}$ to obtain a Fourier series 
\begin{equation}\label{fourier-f-intro}
	u(t,x) = \sum_{j \in \N} u_j(t) \varphi_j(x),
\end{equation}
for $u$ as in \eqref{def-periodic-gelf-intro} and $u_j(t)$ a sequence of periodic functions.

It is important to emphasize that a similar approach is presented in the paper \cite{AvGrKi} where the authors characterize the functional spaces $C^{\infty}(\T \times M)$ and $\D(\T \times M)$, $M$ being a compact manifold, in terms of Fourier expansions generated by an elliptic operator on $M$.  In particular, they study the $C^{\infty}$-global hypoellipticity  for operators of type $\mathscr{L} = D_t + C(t,x,D_x)$,  where  $C(t,x,D_x)$ is a pseudo-differential operator on $M$, smooth in $t$, (see also \cite{AvKi}).
We recall that Hounie presented in \cite{Hou79} the study of global properties of the abstract operator  $\mathcal{L} = D_t + b(t) A + r(t,A)$,  where   $A$ is a linear self-adjoint operator, densely defined in a separable complex Hilbert space $H$ which is unbounded, positive, and has eigenvalues diverging to $+\infty$, and $r(t,A)$ is a lower order term in a suitable sense. It was observed by Hounie the remarkable fact that the structure of the spectrum of $A$ does not  interfere on the solvability and hypoellipticity in case where $t$ belongs to some interval in $\R$, while its  behavior  in a neighborhood of infinity plays an important role when $t \in \T$.

In this paper, as an application of definition \eqref{def-periodic-gelf-intro}   and expansion \eqref{fourier-f-intro} we present the study of \textit{global hypoellipticity} (see Definition \ref{def-GH-2}) for the operator 
\begin{equation}\label{op-intro}
	L= D_t + c(t)P(x,D_x), \  D_t = i^{-1}\partial_t, 
\end{equation}
where $c=c(t)$ belongs to some Gevrey class on $\T$ and $P(x, D_x)$ is a normal differential operator satisfying a suitable global ellipticity condition. Namely,  expansions \eqref{fourier-f-intro} lead us to study an equation $Lu=f$ by means of a discretization process, namely, in terms of the sequence of ordinary differential equations
$$
D_tu_j(t) + \lambda_j c(t) u_j(t) = f_j(t), \ t \in \T, \ j \in \N,
$$
where $\{\lambda_j\}$ is the sequence of eigenvalues of $P(x,D_x)$.
Hence, the regularity of solution $u$ is studied by analyzing the behavior  of the sequence $u_j(t)$ (and all its derivatives), as $j\to \infty$, cf. Theorem \ref{charGSspaces} below.

The work is organized as follows: in Section \ref{sec-periodic-gelf}, we introduce  the \textit{time-periodic Gelfand-Shilov space} $\mathcal{S}_{\sigma, \mu}(\mathbb{T}\times \R^n)$, its dual 
$\mathcal{S}_{\sigma, \mu}'(\mathbb{T}\times \R^n)$, and the corresponding topologies.  In Subsection \ref{sec-eign-func-exps} we explore properties of these spaces in terms of a Fourier expansion on the variable $x$ generated by $\{\varphi_j(x)\}_{j \in \N}$. In particular, in Proposition \ref{propequivPM}, we obtain a characterization for $\mathcal{S}_{\sigma, \mu}(\mathbb{T}\times \R^n)$ in terms of powers of the elliptic operator $P$. In turn, with Theorems \ref{charGSspaces} and \ref{theorm-seq}  the exact meaning of the expression \eqref{fourier-f-intro}  is formalized.
In Section \ref{sec-hypoellipticity}, we present a complete analysis of the global hypoellipticity for the operator \eqref{op-intro}. The starting point is an approach for the time independent coefficients operator
$D_t + (\alpha+i \beta)P(x,D_x)$. In Theorem  \ref{main_const} we obtain that the hypoellipticity is connected with  \textit{Diophantine type phenomena}.
We observe that the presence of  Diophantine approximations 
in the investigations for global hypoellipticity (and global solvability) has been observed in \cite{AvGrKi,Hou79}  and is also  widely explored in the context of operators on the torus, as the reader can see in \cite{BDG18,AGKM,DGY97,GW1,GW3,HiPetr00,PETRONILHO2005} and the references therein.
In Subsection \ref{sec-variable_coeff}, the case when the coefficients depend on $t$ is studied and the  main result is Theorem \ref{main-variable} which states the following:  $L$  is globally hypoelliptic if and only if either $\Im c(\cdot)$ does not change sign, or 
$\Im c(t) \equiv 0$ and 
$$
\kappa = (2\pi)^{-1} \int_{0}^{2\pi} \Re c(t)dt
$$
is not a $\mu$-exponential Liouville number with respect to the sequence $\{\lambda_{j}\}$, as defined in \eqref{non-exp-Liou}. 

We emphasize that Theorem \ref{main-variable}
extends part of the studies in \cite{AvGrKi} to the non-compact case 
$\T \times \R^n$. It is also an improvement for analysis presented in
\cite{Hou79} in the following sense: the global hypoellipticity 
in that work  is given in terms of Sobolev scales generated by the elliptic operator. However, since we are modeling our work on the Gelfand-Shilov setting we are not only studying the regularity of solutions, but also their decay at infinity.  Moreover, with this approach it is possible to measure the loss of regularity for the solutions of equation $Lu=f$, as observed in Remarks \ref{remark-lost-constant} and \ref{remark-loss-variable}.


\section{The spaces  $\mathcal{S}_{\sigma, \mu}(\mathbb{T}\times \R^n)$ and $\mathcal{S}_{\sigma, \mu}'(\mathbb{T}\times \R^n)$ \label{sec-periodic-gelf}}

In this section, we introduce the  spaces under investigation. Let  $\sigma \geq 1$ and $\mu\geq 1/2$ be fixed. Given a constant $C>0$ we define by  $\mathcal{S}_{\sigma,\mu,C} (\mathbb{T}\times \R^n)$ the space of all smooth functions on $\mathbb{T}\times \R^n$ such that 
$$
\|\varphi\|_{\sigma,\mu,C} := \sup_{\alpha, \beta \in \N^n, j \in \N}C^{-|\alpha+\beta|-j}j!^{-\sigma} (\alpha!\beta!)^{-\mu}\sup_{(t,x) \in \mathbb{T}\times \R^n} |x^\alpha \partial_x^\beta \partial_t^j u(t,x)| <+\infty.
$$

It is easy to verify that the space
$\mathcal{S}_{\sigma,\mu,C} (\mathbb{T}\times \R^n)$ is a Banach space endowed with the norm $\|\varphi\|_{\sigma,\mu,C}$.
Therefore, we can define
\begin{equation}\label{def-S}
\mathcal{S}_{\sigma,\mu} (\mathbb{T}\times \R^n) = \bigcup_{C> 0}
\mathcal{S}_{\sigma,\mu,C} (\mathbb{T}\times \R^n),
\end{equation}
and equip it with the inductive limit topology:  
$$
\mathcal{S}_{\sigma,\mu} (\mathbb{T}\times \R^n)  =\displaystyle \underset{C\rightarrow +\infty}{\mbox{ind} \lim} \;\mathcal{S}_{\sigma,\mu,C}(\mathbb{T}\times \R^n).
$$
%
In particular, a sequence $\{\varphi_j\}_{j \in\N}$ of elements of
$\mathcal{S}_{\sigma,\mu} (\mathbb{T}\times \R^n)$ converges to $\varphi$ in $\mathcal{S}_{\sigma,\mu}  (\mathbb{T}\times \R^n)$ if and only if there exists $C>0$ such that $\varphi,\varphi_j \in \mathcal{S}_{\mu,\sigma, C }(\mathbb{T}\times \R^n)$ for all $j \in \N$ and
$$
\|\varphi_j - \varphi\|_{\sigma, \mu, C} \to 0, \ \textrm{ as } \ j \to \infty. 
$$

\begin{definition}\label{dualspace} Given $\mu \geq 1/2, \sigma \geq 1$, we denote by 
$\mathcal{S}_{\sigma, \mu}'(\mathbb{T}\times \R^n)$ the space of all linear forms $u: \mathcal{S}_{\sigma, \mu} (\mathbb{T}\times \R^n) \to \C$ satisfying the following condition: for every $A>0$ there exists $C=C(A)$ such that
\begin{equation}\label{def-ultra}
	|\langle u , \varphi\rangle| \leq C \sup_{\alpha, \beta, k}A^{-|\alpha+\beta|-k} k!^{-\sigma}(\alpha! \beta!)^{-\mu}\sup_{\mathbb{T}\times \R^n} |x^\alpha \partial_x^\beta \partial_t^k \varphi(t,x)|
\end{equation}
for every $\varphi \in \mathcal{S}_{\sigma, \mu} (\mathbb{T}\times \R^n)$. \end{definition}

\begin{remark}
In order to not overload our notations we may write  $\mathcal{S}_{\sigma,\mu,C}$, $\mathcal{S}_{\sigma, \mu}$ and $\mathcal{S}_{\sigma, \mu}'$, that is, omitting $\T\times\R^n$ in the notation.
\end{remark}

Now, let us recall the standard characterization of the Gevrey classes on the torus $\T$. Given $h>0$ and $\sigma \geq 1$, we denote by $\mathcal{G}^{\sigma,h}(\T)$  the space of all smooth functions $\varphi \in C^{\infty}(\T)$ such that there exists $C>0$ for which 
\begin{equation}\label{def-gsh}
	\sup_{t \in \T} |\partial^{k}\varphi(t)| \leq C h^{k}(k!)^{\sigma}, \ \forall k \in \Z_+.
\end{equation}

The space $\mathcal{G}^{\sigma,h}(\T)$ is a Banach space endowed with the norm
\begin{equation}\label{def-ngsh}
	\|\varphi\|_{\sigma,h} = \sup_{k \in \Z_+}\left \{\sup_{t \in \T} |\partial^{k}\varphi(t)|  h^{-k}(k!)^{-\sigma}\right\},
\end{equation}
thus the space of periodic Gevrey functions of order $\sigma$ is defined by
$$
\mathcal{G}^{\sigma}(\T) =\displaystyle \underset{h\rightarrow +\infty}{\mbox{ind} \lim} \;\mathcal{G}^{\sigma, h} (\mathbb{T}).
$$

The dual space $(\mathcal{G}^{\sigma})'(\T)$ is the set of all linear maps $u: \mathcal{G}^{\sigma}(\T) \to \C$ satisfying the following: for every $h>0$ there exists $C_h>0$ such that
\begin{equation}\label{de-dual-s-1}
	|\langle u , \psi \rangle| \leq C_h
	\sup_{t\in \T, k \in \Z_+} h^{-k} (k!)^{-\sigma}
	|\partial^k \psi(t)|, \ \forall \psi \in \mathcal{G}^{\sigma}(\T),
\end{equation}
or equivalently, for any $f \in \mathcal{G}^{\sigma}(\T)$ we obtain for every $h>0$ a positive constant $C_h$ such that
\begin{equation}\label{de-dual-s-2}
	|\langle u , f\rangle| \leq C_h \|f\|_{\sigma, h}.
\end{equation}


\subsection{Fourier expansion for the spaces $\mathcal{S}_{\sigma, \mu}(\mathbb{T}\times \R^n)$ and $\mathcal{S}'_{\sigma, \mu}(\mathbb{T}\times \R^n)$}
\label{sec-eign-func-exps}

Motivated by the approach presented in \cite{GPR}, fixed $m \geq 2$, we consider a partial differential operator  $P$ of the form
\begin{equation}\label{P-intro}
	P = P(x,D) = \sum_{|\alpha| + |\beta| \leq m} c_{\alpha,\beta} x^{\beta} D^{\alpha}_x, \ D^{\alpha}_x = (-i)^{|\alpha|}\partial_x^{\alpha},
\end{equation}
with $c_{\alpha,\beta} \in \C,$ satisfying the normal condition $P^*P = PP^*$ and the global ellipticity property 
\begin{equation}\label{P-elliptic}
	p_m(x,\xi) = \sum_{|\alpha| + |\beta| = m} c_{\alpha,\beta} x^{\beta} \xi^{\alpha} \neq 0, \quad (x,\xi) \neq (0,0).
\end{equation}

Notice that for $m=1$ the condition \eqref{P-elliptic} is never satisfied by an operator of the form \eqref{P-intro}. This justifies the assumption $m \geq 2$ above. Under these conditions there exists an orthonormal basis of eigenfunctions $\{\varphi_j\}_{j \in \N} \subset \mathcal{S}^{1/2}_{1/2}(\R^n)$ with real eigenvalues $\lambda_j$ such that $|\lambda_j| \to \infty$. Moreover, we have the Weyl asymptotic formula
\begin{equation}\label{weyl}
	|\lambda_j| \sim C j^{\, m/2n}, \ \textrm{ as } \ j \to \infty,
\end{equation}
for some positive constant $C$. In particular, any $u \in L^2(\R^n)$ (or $\mathscr{S}' (\R^n)$) can be expanded as
$$
u = \sum_{j \in \N} u_j \varphi_j,
$$
where $u_j = (u , \varphi_j )_{L^2(\R^n)}, \ j \in \N$
(or $u_j = \langle u , \varphi_j\rangle$).

Moreover, it follows from Theorem 1.2 in \cite{GPR} that $u \in \mathcal{S}^{\mu}_{\mu}(\R^n)$,  $\mu\geq 1/2$, if and only if there exists $\epsilon>0$ such that
$$
\sum_{j \in \N} |u_j|^2 \exp(\epsilon |\lambda_j|^{\frac{1}{m\mu}}) < \infty,
$$
or equivalently, there exist $\epsilon,C>0$ such that
$$
|u_j| \leq C \exp(-\epsilon j^{\frac{1}{2n\mu}}), \ j \in \N.
$$

In the sequel we present an extensive characterization of $\mathcal{S}_{\sigma,\mu}(\mathbb{T}\times \R^n)$ and 
$\mathcal{S}_{\sigma,\mu}'(\mathbb{T}\times \R^n)$ in terms of the Fourier analysis generated by $\{\varphi_j \}$. First of all it is easy to show that
$f \in \mathcal{S}_{\sigma,\mu} (\mathbb{T}\times \R^n)$ if and only if there exist $A,C >0$ such that
\begin{equation}\label{equivL^2}
	\sup_{t \in \mathbb{T}}\sum_{|\alpha+\beta| =s} \| x^\beta D_x^\alpha \partial_t^k f \|_{L^2(\R^n_x)} \leq CA^{s+k} s!^\mu k!^\sigma
\end{equation}
for every $s,k \in \N$.

\begin{proposition}\label{propequivPM}
	Let $P$ be an operator of order $m$ of the form \eqref{P-intro} satisfying \eqref{P-elliptic}. Let $\mu \geq 1/2$ and $u \in \mathcal{S}_{\sigma,\mu}' (\mathbb{T}\times \R^n)$. Then $u \in \mathcal{S}_{\sigma,\mu} (\mathbb{T}\times \R^n)$ if and only if
	there exists $C>0$ such that
	\begin{equation}\label{equivPM}
		\| P^M \partial_t^k u (t, \cdot) \|_{L^2(\R^n_x)} \leq C^{M+k+1}(M!)^{\mu m}k!^\sigma 
	\end{equation}
	for every $k, M \in \N$.
\end{proposition}

\begin{proof}
	The proof is a variant of the proof of the analogous characterization for the space $\mathcal{S}^{\mu}_\mu (\R^n)$ given in \cite[Lemma 3.1]{GPR} so we just sketch it. Assume that $ u \in \mathcal{S}_{\sigma,\mu} (\mathbb{T}\times \R^n)$. Then using \eqref{equivL^2} and observing that for every $M \in \N$ the iterated operator $P^M $ is of the form 
	$$P^M = \sum_{|\alpha+\beta| \leq mM} c'_{\alpha, \beta} x^\beta D_x^\alpha$$ for some $c'_{\alpha, \beta }\in \C$ such that
	$\sup_{|\alpha+\beta|\leq mM}|c'_{\alpha,\beta}| \leq C^M$ for some positive constant $C$ independent of $M$, we get
	\begin{eqnarray*} \|P^M \partial_t^k u(t, \cdot) \|_{L^2(\R^n_x)} &\leq& \sum_{s=0}^{mM}\sum_{|\alpha+\beta| =s} |c'_{\alpha, \beta}|  \|x^\beta D_x^\alpha \partial_t^k u (t, \cdot) \|_{L^2(\R^n_x)} \\
		&\leq& C^M \sum_{s=0}^{mM} CA^{s+k} s!^\mu k!^\sigma \\ &\leq& C_1 A_1^{M+k} (M!)^{\mu m} k!^\sigma,
	\end{eqnarray*}
	hence we obtain \eqref{equivPM}.

	Viceversa, assume that \eqref{equivPM} holds for every $M,k \in \N$. Then in particular $\partial_t^k u(t , \cdot) \in \mathscr{S}(\R^n)$ for every fixed $t \in \mathbb{T}$ and for every $k \in \N.$ Denote 
	\begin{equation}\label{s}| \partial_t^k u(t, \cdot) |_s := \sum_{|\alpha+\beta| =s} \| x^\beta D_x^\alpha \partial_t^k  u (t, \cdot) \|_{L^2(\R^n_x)}.
	\end{equation}
	To conclude it is sufficient to show that there exists $C>0$ such that
	\begin{equation}\label{stimas}
		\sup_{t \in \mathbb{T}}| \partial_t^k u(t, \cdot) |_s \leq C^{k+s+1}s!^{\mu}k!^{\sigma} \qquad \forall s \in \N
	\end{equation}

	Now, by  \cite[Proposition 4.1]{GPR} there exists $C>0$ such that for every $s \in \N$ with $s=pm+r, p \in \N, 0<r<m$ and for all $\varepsilon >0$
	
	\begin{equation}
		|v|_s \leq \varepsilon | v|_{(p+1)m}+C \varepsilon^{-\frac{r}{m-r}}|v|_{pm}+C^s (s!)^{1/2}\| v \|_{L^2}
	\end{equation}
	for every $v \in \mathscr{S}(\R^n)$. This means that it is sufficient to prove \eqref{stimas} for $v= \partial_t^k u $ and for the integers $s=pm$. Fixed $\lambda >0, k \in \N$ we denote
	$$\sigma_{p}(\partial^k_t u,\lambda)= (pm)!^{-\mu} \lambda^{-p} |\partial^k_t u|_{pm}, \qquad p \in \N.$$ We observe that $\sigma_0(\partial^k_t u,\lambda)= \| \partial_t^k u\|_{L^2(\R^n_x)}$ which does not depend on $\lambda.$
	By \cite[Lemma 2.4]{CGPR}, there exists $\lambda_0$ such that for every $\lambda \geq \lambda_0$ we have
	\begin{eqnarray*}\sigma_{p}(\partial_t^k u,\lambda) &\leq& 2^p \sigma_0( \partial_t^k u, \lambda)+\sum_{\ell=1}^p 2^{p-\ell}\binom{p}{\ell}(\ell!)^{-m\mu}\sigma_0(P^\ell \partial_t^k u, \lambda).
	\end{eqnarray*}
	Using \eqref{equivPM} we obtain
	$$\sigma_{p}(\partial^k_t u,\lambda) \leq C^{k+1}k!^\sigma 2^p \left(1+\sum_{\ell=1}^p 2^{-\ell} \binom{p}{\ell}C^{\ell}\right) \leq C_1^{k+p+1} k!^\sigma.$$
	Therefore
	$$|\partial_t^k u|_{pm} \leq C_2^{p+k+1} (pm)!^{\mu}k!^\sigma$$ for some constant $C_2>0$ independent of $p$ and $k$, that is the estimate \eqref{stimas} for $s=pm.$ This concludes the proof.
	
\end{proof}

\begin{theorem} \label{charGSspaces}
	Let $\mu\geq 1/2$ and $\sigma \geq 1$ and let $u \in \mathcal{S}_{\sigma,\mu}' (\mathbb{T}\times \R^n)$. Then $u \in \mathcal{S}_{\sigma,\mu}(\mathbb{T}\times \R^n)$ if and only if 
	it can be represented as 
	\begin{equation}\label{fourier-f}
		u(t,x) = \sum_{j \in \N} u_j(t) \varphi_j(x)
	\end{equation}
	for some $\{u_{j}\}_{j \in \N} \in \mathcal{G}^{\sigma}(\T)$ satisfying the following condition:
	there exist $C >0$ and $\epsilon>0$ such that
	\begin{equation} \label{deccoeff}
		\sup_{t \in \T} | \partial_t^k u_j(t)| \leq
		C^{k+1} (k!)^{\sigma} \exp \left[-\epsilon j^{\frac{1}{2n\mu}} \right] \ \forall j,k \in \N.
	\end{equation}
\end{theorem}
 
To prove Theorem \ref{charGSspaces} we will need the following preliminary result.

\begin{lemma}\label{lemma-exp-j}
	Let $s, \gamma$ be positive numbers and $\ell \in \Z_+$. For every  $\epsilon>0$ there exist $C_{\epsilon}>0$ such that
	\begin{equation}\label{lemma-exp-j-eq}
		j^{\ell \gamma} \exp\left(-\epsilon j^{1/s}\right) \leq   C_{\epsilon}^{\ell} (\ell!)^{s \gamma},
	\end{equation}
	for every $j \in \N$.
\end{lemma}

	
	

\textit{Proof of Theorem \ref{charGSspaces}.}
	Assume that $u \in \mathcal{S}_{\sigma,\mu} (\mathbb{T}\times \R^n).$ Then, for any fixed $t \in \mathbb{T}$ and $k \in \N$, we have $\partial_t^k u (t, \cdot) \in \mathcal{S}^{\mu}_\mu (\R^n).$ Then we have
	$$\partial_t^k u(t,x) = \sum_{j \in \N} a_j^k(t) \varphi_j(x),$$
	where 
	$$a_j^k(t)= (\partial_t^k u(t,\cdot), \varphi_j(\cdot))_{L^2(\R^n_x)} =\partial_t^k (u(t,\cdot), \varphi_j(\cdot))_{L^2(\R^n_x)}, \quad t \in \mathbb{T}.$$
	Moreover, we observe that
	$$\|P^M \partial_t^k u (t,\cdot)\|_{L^2(\R^n_x)}^2 = \left\| \sum_{j \in \N} \lambda_j^M a_j^k(t) \varphi_j(\cdot) \right\|_{L^2(\R^n_x)}^2= \sum_{j \in \N}\lambda_j^{2M}|a_j^k(t)|^2.$$
	
	By Proposition \ref{propequivPM} we get for every $j \in \N, M \in \N:$
	$$\sup_{t \in \mathbb{T}}|a_j^k(t)| \leq C^{M+1+k}k!^\sigma M!^{\mu m}\lambda_j^{-M} \sim C_1^{M+k+1}k!^\sigma M!^{\mu m}j^{-\frac{Mm}{2n}}$$
	for some positive constant $C_1$ independent of $j$ and $M$. Taking the infimum on $M$ of the right hand side, we obtain
	$$\sup_{t \in \mathbb{T}}|a_j^k(t)| \leq C_2^{k+1} k!^\sigma \exp \left[ -\varepsilon j^{\frac{1}{2n\mu}}\right], \quad j,k \in \N.$$
	
	In the opposite direction, assume that $u$ is of the form \eqref{fourier-f} with $u_j$ in $\mathcal{G}^\sigma(\mathbb{T})$ satisfying the condition \eqref{deccoeff}. 
	Then we have $\partial_t^k u \in L^2(\mathbb{T} \times \R^n)$ for every $k \in \N$ and 
	\begin{eqnarray*}\sup_{t \in \mathbb{T}}\|P^M \partial_t^k u(t,\cdot)  \|_{L^2(\R^n_x)} &=& \sup_{t \in \mathbb{T}} \left\| \sum_{j \in \N} \lambda_j^M a_j^k(t) \varphi_j(\cdot) \right\|_{L^2(\R^n_x)} \\
		&\leq&  C^{k+1}k!^\sigma \left[\sum_{j \in \N} j^{\frac{mM}{n}}\exp \left(-2\varepsilon j^{\frac{1}{2n\mu}}\right) \right]^{1/2}
		\\
		&\leq& C_1^{k+M+1}k!^\sigma M!^{m\mu}.
	\end{eqnarray*}
	in view of Lemma \ref{lemma-exp-j}. Then, by Proposition \ref{propequivPM}, $u \in \mathcal{S}_{\sigma,\mu}(\mathbb{T} \times \R^n)$.
	
\qed

\begin{corollary}\label{coro-fourier-coef-est}
	Let $\theta \in  \mathcal{S}_{\sigma,\mu}(\mathbb{T}\times \R^n)$. Then there exists $h>0$ such that $\theta_j \in \mathcal{G}^{\sigma,h}(\T)$ for all $j \in \N$. Moreover, there are $C>0$ and $\epsilon>0$ such that
	\begin{equation}\label{fourier-coef-est}
		\|\theta_j \|_{\sigma,h} \leq C  \exp\left(-\epsilon j^{\frac{1}{2n\mu}}\right), \ \forall j \in\N.
	\end{equation}
\end{corollary}

We now characterize the elements of $\mathcal{S}'_{\sigma,\mu}(\mathbb{T} \times \R^n)$.

\begin{lemma}
Let $u \in \mathcal{S}_{\sigma,\mu}' (\mathbb{T}\times \R^n)$.

\begin{itemize}
	\item [i)] For any $j \in \Z_+$ the linear form $u_j: \mathcal{G}^{\sigma}(\mathbb{T}) \to \C$ given by
	\begin{equation}\label{partial}
		\langle  u_j(t) \, , \, \psi(t) \rangle \doteq 
		\langle  u ,  \psi(t)\varphi_j(x) \rangle
	\end{equation}
	belongs to $(\mathcal{G}^{\sigma})'(\T)$. Moreover, for every $\epsilon>0$ and $h >0$, there exists $C_{\epsilon,h} >0$ such that
\begin{equation} \label{stimamancante}
|\langle  u_j(t) \, , \, \psi(t) \rangle | \leq C_{\epsilon,h} \| \psi\|_{\sigma, h} \exp \left(\epsilon j^{\frac1{2n\mu}}\right) \qquad \forall j \in \N, \, \forall \psi \in \mathcal{G}^{\sigma,h}(\T).
\end{equation}

	\item [ii)]  We may write
	\begin{equation}\label{def_serie}
		\langle  u \, , \, \theta \rangle  =
		\sum_{j \in \N} \langle u_j(t)\varphi_j(x)  \, , \, \theta \rangle
	\end{equation}
	where
	\begin{equation}
		\langle  u_j(t)\varphi_j(x) \, , \, \theta \rangle \doteq 
		\left\langle u_j(t)  \, , \, \int_{\mathbb{R}^n} \theta(t,x) \varphi_j(x) dx \right\rangle 
	\end{equation}
and  $\{u_j(t)\}_{j \in \N} \subset (\mathcal{G}^{\sigma})'(\T)$ is given by \eqref{partial}.
\end{itemize}

\end{lemma}

\begin{proof}
$i)$ Indeed, given $\epsilon>0$ and setting $A = 1/\epsilon$ in \eqref{def-ultra} we obtain
$$
|\langle  u_j(t) \, , \, \psi(t) \rangle| \leq 
\kappa_{\epsilon,j} C_{\epsilon} 
\sup_{t \in \T, k \in \Z_+} \epsilon^{k}(k!)^{-\sigma} |\partial_t^k \psi(t)|,
$$
where
$$
\kappa_{\epsilon,j}  = 
\sup_{\R^n}\sup_{\alpha, \beta} \left(\epsilon^{|\alpha+\beta|} \alpha!^{-\mu} \beta!^{-\mu} |x^\alpha \partial_x^\beta \ \varphi_j(x)|\right),
$$
therefore the first item follows directly from \eqref{de-dual-s-1}. The inequality \eqref{stimamancante} easily follows from Definition \ref{dualspace} and from \cite[Lemma 2.2]{GPR}.

$ii)$ Now, let  $\theta \in \mathcal{S}_{\sigma,\mu}$ and consider its expansion
$$
\theta(t,x) = \sum_{j \in \N}\theta_j(t)\varphi_j(x),
$$
where 
$\theta_j(t) = \int_{\R^n} \theta(t,x) \varphi_j(x)dx.$ Hence, 
\begin{align*}
	\langle u \, , \, \theta(t,x) \rangle & = \sum_{j \in \N}
	\langle u \, , \, \theta_j(t)\varphi_j(x) \rangle 
	= \sum_{j \in \N}  \langle u_j(t) \, , \, \theta_j(t) \rangle \\
	& = \sum_{j \in \N} \left \langle u_j(t) \, , \, \int_{\R^n} \theta(t,x) \varphi_j(x)dx\right\rangle\\
	& = \sum_{j \in \N}\langle u_j(t)\varphi_j(x) \, , \, \theta(t,x) \rangle.
\end{align*}

\end{proof}

\begin{theorem}\label{the-cauchy}
	Let $\{u_j\}_{j \in \N}$ be a sequence in $\mathcal{S}_{\sigma,\mu}' (\mathbb{T}\times \R^n)$ such that
	$\{\langle u_j,\theta \rangle \}_{j \in \N}$ is a Cauchy sequence in $\C$, for all 
	$\theta \in \mathcal{S}_{\sigma,\mu} (\mathbb{T}\times \R^n)$. Then there exists $u \in \mathcal{S}_{\sigma,\mu}' (\mathbb{T}\times \R^n)$ such that
	$$
	\langle u, \theta \rangle = \lim_{j}  \langle u_j, \theta \rangle
	$$
\end{theorem}

\begin{proof}
	For each $\theta \in \mathcal{S}_{\sigma,\mu}$ we define
	$$
	\langle u , \theta \rangle = \lim_{j} \langle u_j , \theta \rangle.
	$$
	
	By hypothesis, $u$ is a linear operator from 
	$ \mathcal{S}_{\sigma,\mu}$ to $\C$. Let us prove that is continuous.  For this, let $\{\varphi_\ell\}_{\ell \in\N}$ be a sequence in $\mathcal{S}_{\sigma,\mu}$ converging to $0$. By the inductive topology we may fix $C>0$ such that
	$\varphi_\ell \in \mathcal{S}_{\sigma,\mu,C}$ and $\varphi_\ell \to 0$ in $\mathcal{S}_{\sigma,\mu,C}$. For each $j \in \N$, we set
	$$
	\omega_j \doteq u_j|_{\mathcal{S}_{\sigma,\mu,C}}: \mathcal{S}_{\sigma,\mu,C}  \to \C,
	$$
	which is linear and continuous.

	Note that, if $\psi \in \mathcal{S}_{\sigma,\mu,C}\subset \mathcal{S}_{\sigma,\mu}$, then
	$\{\langle \omega_j , \psi \rangle\}_{j \in \N}$ is Cauchy sequence in $\C$, hence it is bounded.
%
	Since $\mathcal{S}_{\sigma,\mu,C}$ and $\C$ are Banach spaces we obtain from the Banach-Steinhaus theorem that $\{\omega_j\}_{j \in \N}$  is uniformly bounded on the unitary ball, namely, there is a positive constant
	$K$  such that
	\begin{equation}\label{BS}
		|\langle \omega_j , \psi \rangle| \leq K, \ \|\psi\|_{\sigma, \mu, C} \leq 1, \ \forall  j \in \N,
	\end{equation}
	
	Now, let $\epsilon>0$ and set 
	$$
	\psi_{\ell} = \frac{2 K}{\epsilon} \varphi_\ell.
	$$
	By construction, we have 	$\psi_\ell \in \mathcal{S}_{\sigma,\mu,C}$ and $\psi_\ell \to 0$ in $\mathcal{S}_{\sigma,\mu,C}$. In particular, we may fix $\ell_0 \in \N$ such that $\|\psi_{\ell}\|_{\sigma,\mu, C} \leq 1$, for $\ell > \ell_0$.
	
	Hence, we get
	\begin{equation}\label{C1}
		|\langle u_j , \varphi_\ell \rangle| \leq \dfrac{\epsilon}{2},
	\end{equation}
	for all $\ell \geq \ell_0$ and $j \in\ N$.
	
	Since $\langle u , \varphi_{\ell} \rangle = \lim_{j} \langle u_j , \varphi_{\ell} \rangle$, we obtain for each $\ell\geq\ell_0$ an index $j_{\ell}$ satisfying 
	\begin{equation}\label{C2}
		|\langle u , \varphi_\ell \rangle - \langle u_{j_{\ell}} , \varphi_\ell \rangle | \leq \dfrac{\epsilon}{2}.
	\end{equation}
	
	Finally, for all $\ell\geq\ell_0$, it follows from \eqref{C1} and  \eqref{C2} that
	\begin{align*}
		|\langle u , \varphi_\ell \rangle| & \leq 
		|\langle u , \varphi_\ell \rangle - \langle u_{j_{\ell}} , \varphi_\ell \rangle | +  |\langle u_{j_{\ell}} , \varphi_\ell \rangle| 
		 \leq \epsilon.
	\end{align*}

\end{proof}

\begin{theorem}\label{theorm-seq}
	Let $\{a_j\}_{j \in \N} \subset (\mathcal{G}^{\sigma})'(\T)$ be a sequence satisfying  the following condition: given $\epsilon>0$ and $h>0$, there exists $C_{\epsilon, h}>0$ such that
	\begin{equation}\label{estimete-distr}
		|\langle  a_j \, , \, \psi \rangle| \leq C_{\epsilon, h} \|\psi\|_{\sigma,h}
		\exp\left(\epsilon j^{\frac{1}{2n\mu}}\right), \ \forall j \in \N, \ \forall \psi \in \mathcal{G}^{\sigma,h}(\T).
	\end{equation}
	
	Then
	\begin{equation}\label{fourier-inv}
		u(t,x) = \sum_{j \in \N} a_j \varphi_j(x),
	\end{equation}
	belongs to $\mathcal{S}_{\sigma,\mu}' (\mathbb{T}\times \R^n)$. Moreover, 
	$$
	\langle  a_j \, , \, \psi(t) \rangle  = \langle  u \, , \, \psi(t)\varphi_j(x) \rangle, \ \forall \psi \in  \mathcal{G}^{\sigma}(\T).
	$$
	
	We may use the notation
	$$
	\{a_{j}\}  \rightsquigarrow 
	u \in   \mathcal{S}_{\sigma,\mu}' (\mathbb{T}\times \R^n).
	$$
\end{theorem}
\begin{proof}
	For each $j \in \N$, define
	$$
	S_j = \sum_{k=0}^{j} a_k \varphi_k(x) \ \in \mathcal{S}_{\sigma,\mu}' (\mathbb{T}\times \R^n).
	$$

	Let $\theta \in  \mathcal{S}_{\sigma,\mu}(\mathbb{T}\times \R^n)$. By Corollary \ref{coro-fourier-coef-est} we may consider 
	$h>0$ such that $\theta_j \in \mathcal{G}^{\sigma,h}(\T)$ for all $j \in \N$. Given $\epsilon>0$ we obtain by hypothesis a positive constant
	$C_{\epsilon,h}>0$ such that
	\begin{align*}
		|\langle S_{j+\ell} - S_j , \theta  \rangle| & \leq
		\sum_{k=j}^{\ell} |\langle  a_k \, , \, \theta_k(t) \rangle| \\
		& \leq C_{\epsilon,h} \sum_{k=j}^{\ell} \|\theta_k\|_{\sigma,h}
		\exp\left(\epsilon k ^{\frac{1}{2n\mu}}\right).
	\end{align*}
	
	Now, in view of \eqref{fourier-coef-est} we may find $\epsilon_0>0$ such that
	\begin{equation*}
		\|\theta_k \|_{\sigma,h} \leq C  \exp\left(-\epsilon_0 k^{\frac{1}{2n\mu}}\right), \ \forall k \in\N,
	\end{equation*}
	for some constant $C>0$ independent of $k$.
	
	Hence, for $\epsilon = \epsilon_0/2$, we obtain
	\begin{align*}
		|\langle S_{j+\ell} - S_j , \theta  \rangle| 
		& \leq C_{\epsilon_0,h}C \sum_{k=j}^{\ell}
		\exp\left(-\frac{\epsilon_0}{2} k ^{\frac{1}{2n\mu}}\right),
	\end{align*}
	then $\{S_j(\theta)\}_{j\in\C}$ is a Cauchy sequence in $\C$, for all  $\theta \in  \mathcal{S}_{\sigma,\mu} (\mathbb{T}\times \R^n)$.
	
	It follows from Theorem \ref{the-cauchy} that there exists $u \in \mathcal{S}_{\sigma,\mu}' (\mathbb{T}\times \R^n)$
	such that
	$$
	\langle u , \theta \rangle = \lim_j  \langle  S_j , \theta \rangle,
	\ \forall \theta \in  \mathcal{S}_{\sigma,\mu} (\mathbb{T}\times \R^n).
	$$

\end{proof}

\section{Global hypoellipticity \label{sec-hypoellipticity}}

The main goal of this section is an investigation of the global hypoellipticity for the operators
\begin{equation}\label{general_operator}
	L = D_t + c(t)P(x,D_x), \ t \in \T,\  x \in \R^n, \ D_t = -i\partial_t,
\end{equation}
where $P(x,D_x)$ is a normal operator of the form \eqref{P-intro} satisfying \eqref{P-elliptic} and $c(t)$ belongs to the Gevrey class $\mathcal{G}^{\sigma}(\T)$. Since the arguments in the sequel will use cut-off functions, we restrict to the case $\sigma>1$.

In order to introduce the notion of global hypoellipticity, consider the spaces
$$
\mathscr{F}_\mu(\T \times \R^n) = \bigcup_{\sigma > 1}\mathcal{S}_{\sigma, \mu}(\mathbb{T}\times \R^n)
\	\textrm{ and } \
\mathscr{U}_\mu(\T \times \R^n) = \bigcup_{\sigma > 1}\mathcal{S}_{\sigma, \mu}'(\mathbb{T}\times \R^n).
$$ 

\begin{definition}\label{def-GH-2}
	We say that the operator $L$ is $\mathcal{S}_{\mu}$-globally hypoelliptic on $\T \times \R^n$  ($\mathcal{S}_{\mu}$-GH, for short) if conditions  
	$$
	u \in \mathscr{U}_\mu(\T \times \R^n) \ \textrm{and} \
	Lu \in \mathscr{F}_{\mu}(\T \times \R^n)
	$$
	imply $u \in \mathscr{F}_\mu(\T \times \R^n)$.
\end{definition}

\begin{remark}
We claim that if  $Lu = f \in \mathscr{F}_{\mu}(\T \times \R^n)$, then we may assume without loss of generality that  $f \in \mathcal{S}_{\sigma,\mu}(\T \times \R^n)$, that is, each $f_j$ belongs to the same Gevrey class as the coefficient $c=c(t)$. As a matter of fact, since $f \in \mathscr{F}_{\mu}(\T \times \R^n)$ we get  $f \in \mathcal{S}_{\theta,\mu}(\T \times \R^n)$, for some $\theta > 1$. Therefore:

\begin{itemize}
	\item if $\sigma \leq \theta$, then the claim is a consequence of inclusion $\mathscr{G}^{\sigma}(\T) \subseteq \mathscr{G}^{\theta}(\T)$;
	
	\item for $\theta<\sigma$ we use inclusion $ \mathcal{S}_{\theta,\mu}(\T \times \R^n) \subset \mathcal{S}_{\sigma,\mu}(\T \times \R^n)$.
\end{itemize}
\end{remark}

Now, if  $u \in \mathscr{U}_\mu(\T \times \R^n)$ is such that 
$iLu = f\in \mathcal{S}_{\sigma,\mu}(\T \times \R^n)$, then it follows from the eigenfunction expansions
\begin{equation*}
	u(t,x) = \sum_{j \in \N} u_j(t) \varphi_j(x)
	\ \textrm{ and } \
	f(t,x) = \sum_{j \in \N} f_j(t) \varphi_j(x),
\end{equation*}
that $u_j$ solve the equations
\begin{equation}\label{diffe-equations}
	\partial_t u_j(t) + i \lambda_j c(t) u_j(t) = f_j(t), \ t \in \T, \ j \in \N.
\end{equation}

Setting $c_0 = (2 \pi)^{-1} \int_{0}^{2 \pi} c(t) dt$ 
we obtain, from the ellipticity of equation \eqref{diffe-equations},  the following result.

\begin{proposition} \label{solutions}
	Let $u$ and $f$ as above. Then $u_j$ belongs to $\mathcal{G}^{\sigma}(\mathbb{T}),$ for all $j \in \N$. Moreover, for each $j\in \N$ such that $\lambda_jc_0 \notin \Z$, equation \eqref{diffe-equations} has a unique solution, which can be written in the following equivalent two ways:
	\begin{align}\label{Solu-1}
		u_j(t) = \frac{1}{1 - e^{-  2 \pi i\lambda_j c_0}} \int_{0}^{2\pi}\exp\left(-i\lambda_j\int_{t-s}^{t}c(r) \, dr\right) f_j(t-s)ds, 
	\end{align}
	or
	\begin{align}\label{Solu-2}
		u_j(t) = \frac{1}{e^{ 2 \pi i \lambda_j c_0} - 1} \int_{0}^{2\pi}\exp\left(i\lambda_j\int_{t}^{t+s}c(r) \, dr\right) f_j(t+s)ds. 
	\end{align}
\end{proposition}

From the latter result it follows that the study of global hypoellipticity problem for the operator $L$ can be reduced to the analysis of the behavior of the solutions \eqref{Solu-1} (or \eqref{Solu-2}) as $j \to \infty$. Moreover, 
 $L$ is $\mathcal{S}_{\mu}$-GH if and only if there exists $\theta > 1, C >0$ and $\epsilon>0$ such that
$$
\sup_{t \in \T} | \partial_t^{\gamma} u_j(t)| \leq
C^{\gamma+1} (\gamma!)^{\theta} \exp 
\left(-\epsilon j^{\frac{1}{2n\mu}} \right), \textrm{ as } \ j \to \infty,
$$
for every $\gamma \in \Z_+$.

Clearly, an analysis as suggested by this estimate requires a special attention for  estimates of the derivatives of the exponential terms in the integrals \eqref{Solu-1} (or \eqref{Solu-2}). Moreover, the Diophantine approximations suggested by  $1 - e^{-  2 \pi i\lambda_j c_0}$ play an important role and they are connected with \textrm{time independent} coefficients operators. In view of these considerations, our first investigations are directed to the study of this case and presented in Subsection \ref{sec-const-coeff}. The analysis of operators with time dependent coefficients is developed in Subsection \ref{sec-variable_coeff}.

\subsection{Operators with time independent coefficients} \label{sec-const-coeff}

Let $\mathcal{L}$ be the  operator
\begin{equation} \label{ccop}
\mathcal{L} = D_t + (\alpha + i\beta)P(x,D_x), \ t \in \T, \ \alpha, \beta \in \R.
\end{equation}

Given $f\in \mathcal{S}_{\sigma,\mu}(\T \times \R^n)$ and $u \in \mathscr{U}_\mu(\T \times \R^n)$ satisfying $i\mathcal{L}u = f$ we get the equations
\begin{equation}\label{diffe-equations-constant-case}
	\partial_t u_j(t) + i (\alpha + i \beta)\lambda_j  u_j(t) = f_j(t), \ t \in \T, \ j \in \N.
\end{equation}

It follows from Proposition \ref{solutions} that
$u_j$ belongs to $\mathcal{G}^{\sigma}(\mathbb{T}).$ In particular, if  $j \in \N$ is such that $(\alpha + i \beta)\lambda_j \notin \Z$,
then  the solution of 
\eqref{diffe-equations-constant-case} can be written in the following equivalent two ways:
\begin{equation}\label{Solu-costante-1}
	u_j(t) = \frac{1}{1 - e^{-  2 \pi i (\alpha + i \beta)\lambda_j}} \int_{0}^{2\pi}\exp\left(s (\beta - i\alpha)\lambda_j \right) f_j(t-s)ds,
\end{equation}
or
\begin{equation}\label{Solu-costante-2}
	u_j(t) = \frac{1}{e^{ 2 \pi i (\alpha + i \beta)\lambda_j} - 1} \int_{0}^{2\pi}\exp\left( s (-\beta + i\alpha) \lambda_j \right) f_j(t+s)ds.
\end{equation}

If $\beta\lambda_j \leq 0$ we use formula \eqref{Solu-costante-1}, and for $\beta\lambda_j  \geq 0$ we use \eqref{Solu-costante-2}. Thus, for any given  $\gamma \in \Z_+$ we get 
\begin{equation}\label{der_uj_case1}
	|\partial_t^{\gamma} u_j(t)| \leq 
	2\pi \Theta_j
	\sup_{t \in \T} |\partial_t^{\gamma} f_j(t)|
\end{equation}
for $\beta\lambda_j  \leq 0$ and 
\begin{equation}\label{der_uj_case2}
	|\partial_t^{\gamma} u_j(t)| \leq 2\pi
	\Gamma_j 
	\sup_{t \in \T} |\partial_t^{\gamma} f_j(t)|
\end{equation}
for $\beta\lambda_j  \geq 0$, where 
\begin{equation}\label{Theta_j-Gamma_j}
	\Theta_j = |1 - e^{-  2 \pi i (\alpha + i \beta)\lambda_j}|^{-1}
	\ \textrm{ and} \
	\Gamma_j = |e^{  2 \pi i (\alpha + i \beta)\lambda_j}-1|^{-1}.
\end{equation}

\smallskip

Now, motivated by Definition 3.3 in
\cite{AvGrKi} and Proposition 1.3 in \cite{GPY}, we introduce the following definition.

\begin{definition}\label{def-non-exp-Liou}
We say that a real number $\kappa$ is not a $\mu$-exponential Liouville number with respect to the sequence $\{\lambda_j\}$ if for every $\epsilon>0$ there exists $C_{\epsilon}>0$ such that
\begin{equation}\label{non-exp-Liou}
	\inf_{\tau \in \Z} | \tau - \kappa\lambda_j| \geqslant C_{\epsilon} \exp \left(-\epsilon j^{\frac{1}{2n\mu}}\right), \ \textrm{ as }  \ j \to \infty.
\end{equation}
\end{definition}

\smallskip

We make use of the next lemma,  whose proof can be obtained by a slight modification of the arguments in the proof of  Lemma 2.5 in \cite{BDG18}.

\begin{lemma}\label{lt2} Consider $\eta\geq1$  and $\omega \in \C.$  The following two conditions are equivalent:
	\begin{itemize}
		\item[i)] for each $\epsilon>0$ there exists a positive constant $C_\epsilon$ such that
		$$
		|\tau - \omega\lambda_j|\geqslant C_\epsilon\exp\{-\epsilon(|\tau|+j)^{1/\eta}\}, \ \forall  \tau \in \mathbb{Z},\,  \forall j \in \mathbb{N}.
		$$
		
		\item[ii)] for each $\delta>0$ there exists a positive constant $C_\delta$ such that 
		$$
		|1-e^{2\pi i \omega\lambda_j}|\geqslant C_\delta\exp\{-\delta j^{1/\eta}\}, \ \forall  j \in\mathbb{N}.
		$$
		
	\end{itemize}
\end{lemma}

With the next result we obtain a complete characterization of the global hypoellipticity for operators of the form \eqref{ccop}.

\begin{theorem}\label{main_const}
	The operator $\mathcal{L}$ is $\mathcal{S}_{\mu}$-GH if and only if one of the following conditions holds:
	\begin{enumerate}
		\item [a)] $\beta \neq 0$;
		
		\item [b)] $\beta = 0$ and $\alpha$ is not a $\mu$-exponential Liouville number with respect to the sequence $\{\lambda_j\}$.
	\end{enumerate}

\end{theorem}

\begin{proof}
Let us start with the sufficiency part. First we observe that either if $\beta\neq 0$ or if $\beta=0$ and  $\alpha$ is not a $\mu$-exponential Liouville number with respect to the sequence $\{\lambda_j\}$, then the set
\begin{equation} \label{W}
	\mathcal{W} =\{j \in \N : (\alpha +i\beta) \lambda_j \in \Z \}
\end{equation}
 is finite. Therefore, the solutions of equations \eqref{diffe-equations-constant-case} are given by \eqref{Solu-costante-1} and \eqref{Solu-costante-2}, for $j$ large enough.
 
Assume  that $\beta\neq 0$.  In this case $\lim_j \Theta_j =1$, when $\beta <0$ and $\lim_j \Gamma_j =1$ for  $\beta >0$, where  $\Theta_j$ and $\Gamma_j$ are given by \eqref{Theta_j-Gamma_j}. Hence, 
$$
|\partial^{\gamma}_t u_j(t)| \leq C \sup_{t \in \T}|
\partial^{\gamma} f_j(t)|,
$$
and $u$ belongs to the same class as $f$. Then $\mathcal{L}$ is $\mathcal{S}_{\mu}$-GH.

Now, assume that $\beta  = 0$ and $\alpha$ is not a $\mu$-exponential Liouville number. Under these assumptions it is enough to consider expression \eqref{Solu-costante-1}. In particular, inequalities \eqref{der_uj_case1} can be  rewritten as 
\begin{equation}
	|\partial_t^{\gamma} u_j(t)| \leq 2\pi 
	\Theta_j
	\sup_{t \in \T} |\partial_t^{\gamma} f_j(t)|.
\end{equation}

By hypothesis and Lemma \ref{lt2} for  every $\epsilon>0$ there exists a constant  $C_{\epsilon}>0$ such that
$$
|1-e^{- 2\pi i \alpha\lambda_j}| \geqslant C_{\epsilon} \exp \left(-\epsilon j^{\frac{1}{2n\mu}}\right), \ \textrm{ as }  \ j \to \infty.
$$

Assume now that $f \in \mathcal{S}_{\sigma,\mu}$ and let $\gamma \in \Z^+$. There exists  $\epsilon_0 >0$ such that 
$$
\sup_{t \in \T} | \partial_t^\gamma f_j(t)| \leq
C^{\gamma+1} (\gamma!)^{\sigma} \exp \left[-\epsilon_0 j^{\frac{1}{2n\mu}} \right].
$$

By fixing $\epsilon = \epsilon_0/2$ we obtain $C_{\epsilon_0}>0$ for which
\begin{align*}
	|\partial_t^{\gamma} u_j(t)| & \leq C_{\epsilon_0}  \exp \left[ \frac{\epsilon_0}{2} j^{\frac{1}{2n\mu}} \right] \sup_{t \in \T} |\partial_t^{\gamma} f_j(t)| \\
	& \leq C_{\epsilon_0} C^{\gamma+1} (\gamma!)^{\sigma} 
	\exp \left[-\frac{\epsilon_0}{2}  j^{\frac{1}{2n\mu}} \right],
\end{align*}
which implies that  $u \in \mathcal{S}_{\sigma,\mu}$ in view of Theorem \ref{charGSspaces}. Therefore $\mathcal{L}$ is $\mathcal{S}_{\mu}$-GH.

To prove the necessary part, we assume that $\beta =0$ and \eqref{non-exp-Liou} fails and exhibit a singular solution to the operator  $\mathcal{L}$.  To do this, note that if \eqref{non-exp-Liou} fails, then there exists $\epsilon'$ and an increasing sequence
$\{\tau_{\ell}\}_{\ell \in \N} \subset \Z$  such that
$$
|\tau_{\ell} - \alpha\lambda_{j_\ell}| <  \exp \left(-\epsilon' j_\ell^{\frac{1}{2n\mu}}\right).
$$

Consider the sequences $\{u_j\}_{j \in \N}, \{f_j\}_{j \in \N} \subset C^{\infty}(\T)$ defined by
$$
u_j(t) = 
\left\{
\begin{array}{l}
	e^{-i\tau_{\ell} t}, \ \textrm{ if } \ j = j_\ell, \\
	0, \ \textrm{ if } \ j \neq j_\ell.
\end{array}
\right.
\ \textrm{ and } \
f_j(t) = 
\left\{
\begin{array}{l}
	(\tau_{\ell} - \alpha\lambda_{j_\ell}) e^{-i\tau_{\ell}t}, \ \textrm{ if } \ j = j_\ell, \\
	0, \ \textrm{ if } \ j \neq j_\ell.
\end{array}
\right.
$$

Since $|\tau_{\ell}|\leq 1 + |\alpha| \lambda_{j_{\ell}}$ we get
\begin{equation}\label{est-tau}
|\tau_{\ell}|^{\gamma}  \leq C \sum_{\beta=0}^{\gamma}\binom{\gamma}{\beta}
|\alpha|^{\beta} j_{\ell}^{\beta m/2n} 
\leq C j_{\ell}^{\gamma m/2n}, 	
\end{equation}
as $\ell \to \infty$. Now, by Lemma \ref{lemma-exp-j}, we obtain  $C_{\epsilon'}>0$ such that
\begin{equation}\label{lemma-exp-jcons}
	j^{\beta m/2n}_{\ell} \exp\left(-\frac{\epsilon'}{2} j^{\frac{1}{2n\mu}}\right)\leq   C_{\epsilon'}^{\beta} (j_{\ell}!)^{m\mu},
\end{equation}
hence 
\begin{align*}
	|\partial_t^{\gamma} f_{j_{\ell}}(t)| & \leq 
	C_{\epsilon'}^\gamma \sum_{\beta=0}^{\gamma} \left[\binom{\gamma}{\beta}
	|\alpha|^{\beta} j_{\ell}^{\beta m/2n}\right] \exp \left(-{\epsilon'} j_\ell^{\frac{1}{2n\mu}}\right) \\
	& \leq C_{\epsilon'}^\gamma (\gamma!)^{ m \mu}\exp\left(-\frac{\epsilon'}{2} j_\ell^{\frac{1}{2n\mu}}\right).
\end{align*}

Therefore we have 
$$
\{u_{\ell}(t)\}  \rightsquigarrow 
u \in \mathscr{U}_\mu(\T \times \R^n)\setminus \mathscr{F}_\mu(\T \times \R^n)
$$
and
$\{f_{\ell}(t)\}  \rightsquigarrow f \in  \mathscr{F}_\mu(\T \times \R^n)$. Since  $\mathcal{L} u  =f$, then $\mathcal{L}$ is not $\mathcal{S}_{\mu}$-GH.

\end{proof}

\begin{remark}\label{remark-lost-constant}
It is important to emphasize that for globally hypoelliptic operators with time independent coefficients there is no loss of regularity on the variable $t$, that is, if $\mathcal{L}$ is $\mathcal{S}_\mu$-GH and $\mathcal{L}u \in \mathcal{S}_{\sigma, \mu}$, then $u \in \mathcal{S}_{\sigma, \mu}$. This is in contrast with the time dependent coefficients case, as the reader can see in Theorem \ref{variable-case} and Remark \ref{remark-loss-variable}.
\end{remark}

\begin{example}[Harmonic oscillator  on $\R$]\label{Harmonic_constant}
Consider on $\T \times \R$ the operator	
$$
\mathcal{L} = D_t + \alpha H, \ \alpha \in \R, 
$$	
where $H$ denote the Harmonic oscillator 
\begin{equation}\label{harmonic-R}
H = -\dfrac{d^2}{dx^2} + x^2, \ x \in \R.
\end{equation}
It is already known from \cite[Proposition 1.3]{GPY} that if $\alpha \notin \Q$ and is not a $2\mu$-exponential Liouville number, then $\mathcal{L}$ is  $\mathcal{S}_{\mu}$-GH.
Since the eigenvalues of $H$ are given by  $\lambda_j = 2j + 1$, $j \in \N_0$, Theorem \ref{main_const} states that $\mathcal{L}$ is  $\mathcal{S}_{\mu}$-GH if and only if for every $\epsilon>0$ there exists $C_{\epsilon}>0$ such that
$$
\inf_{\tau \in \Z} |\tau - \alpha (2j + 1) | \geq C_{\epsilon}\exp
 \left(-\epsilon j^{\frac{1}{2 \mu}}\right), \ j \to \infty.
$$
In particular, it is evident from the latter formula that $\mathcal{L}$ is not $\mathcal{S}_{\mu}$-GH when $\alpha \in \Z$.
\end{example}

As an immediate consequence of Theorem \ref{main_const} we obtain the following  necessary condition for the global hypoellipticity.

\begin{corollary} \label{nec_cond}
	If  $\mathcal{L}$ is $\mathcal{S}_{\mu}$-GH, then the set
	$\mathcal{W}$ defined by \eqref{W}
	is finite.
\end{corollary}

\subsection{Operators with time dependent coefficients  \label{sec-variable_coeff}}

Let us turn back our attention to the general operator 
$$
L = D_t + c(t)P(x,D_x)
$$
as defined in \eqref{general_operator}, where $c(t) = a(t) + ib(t)$ for some real valued functions $a,b \in \mathcal{G}^\sigma (\mathbb{T})$. Set $c_0 = a_0 + i b_0$, where
$$
a_0 = (2 \pi)^{-1} \int_{0}^{2 \pi} a(t) dt \ \textrm{ and } \ 
b_0 = (2 \pi)^{-1} \int_{0}^{2 \pi} b(t) dt.
$$
First of all we extend the necessary condition stated in Corollary \ref{nec_cond} to the case of time dependent coefficients.

\begin{proposition}\label{Z_finite}
	If the operator $L$ is $\mathcal{S}_{\mu}$-globally hypoelliptic, then the set
	 $\mathcal{Z} = \{j \in \mathbb{N}; \, \lambda_j c_0 \in\mathbb{Z} \}$  is finite.
\end{proposition}

\begin{proof} If $\mathcal{Z}$ is infinite, then there exists an increasing sequence $\{j_{\ell}\}_{\ell \in \N}$ such that $c_0 \lambda_{j_{\ell}} \in \mathbb{Z}.$
	Set
	$$
	c_{\ell}=\exp\left({- \lambda_{j_{\ell}} \int_{0}^{t_\ell}\Im c(r)dr}\right),
	$$	
	where $t_\ell\in[0,2\pi]$ satisfies
	$$
	\int_{0}^{t_\ell} \lambda_{j_{\ell}} \Im c(r)dr=\max_{t\in[0,2\pi]}
	\int_{0}^{t} \lambda_{j_{\ell}} \Im c(r)dr.
	$$
	
	For each $j_{\ell}$ the function
	$$
	{u}_{\ell}(t)=c_\ell \exp\left({-i \lambda_{j_{\ell}} \int_{0}^{t}c(r)dr}\right)
	$$
	belongs to $\mathscr{G}^{\sigma}(\T)$ and satisfies the equation
	\begin{equation*}
		\partial_t{u}_{\ell}(t)+ic(t)\lambda_{j_{\ell}}
		{u}_{\ell}(t)=0.
	\end{equation*}
	
	Moreover, $|{u}_{\ell}(t)|\leqslant 1,$ for all $t\in[0,2\pi],$ and $|u_{\ell}(t_{\ell})|=1.$ Hence,
	\begin{equation*}
		\{u_{\ell}(t)\}  \rightsquigarrow u \in  \mathscr{U}_\mu(\T \times \R^n)\setminus \mathscr{F}_\mu(\T \times \R^n),
	\end{equation*}
	and $Lu=0.$ Therefore, $L$ is not $\mathcal{S}_{\mu}$-GH. 
	
\end{proof}

\begin{theorem}\label{main-variable}
The operator $L$ is $\mathcal{S}_{\mu}$-globally hypoelliptic if and only if one of the following conditions holds:

\begin{enumerate}
	\item [a)] $b$ is not identically zero and $b$ does not change sign;
	
	\item [b)] $b \equiv 0$ and $a_0$ is not a $\mu$-exponential Liouville number with respect to sequence $\{\lambda_{j}\}$.
\end{enumerate}
\end{theorem}

The proof of this result will follow by combining Theorems  \ref{variable-case}, \ref{LGHimpliesL_0GH}, \ref {change_sign} and Proposition \ref{L_0GHimpliesLGH} presented in the sequel.

\smallskip

Our next step is the analysis of sufficiency part of item $a)$ in Theorem \ref{main-variable}.

\begin{theorem}\label{variable-case}
	If $b(\cdot)$ does not change sign and is not identically zero, then $L$ is $\mathcal{S}_{\mu}$-globally hypoelliptic.
\end{theorem}

\begin{proof}  By hypothesis we have $b_0 \neq0$, then the set $\mathcal{Z}$ is finite and we may consider expressions  \eqref{Solu-1}	and \eqref{Solu-2}, for $j$ large enough. Moreover,  there exist positive constants 
	$C_1$ and $C_2$ such that
	$$
	0 < C_1 \leq \Gamma_j \leq 1  \ \textrm{ and }  0 < C_2 \leq \Theta_j \leq 1,
	$$
	for $\Theta_j$ and $\Gamma_j$ as defined in \eqref{Theta_j-Gamma_j}. Also, let us admit that  $\lambda_j>0$, for all $j \in \N$. Otherwise, we may interchange the use of   \eqref{Solu-1}	and \eqref{Solu-2}.
	
	Now, by assuming  $b(\cdot) \leq 0$ we may consider expression \eqref{Solu-1}. Set 
	$$
	\mathcal{H}(t,s) = \exp \left( -i\lambda_j \int_{t-s}^t c(r)\, dr \right).
	$$
	It follows from Leibniz rule that
\begin{align*}
\partial_t^{\gamma} u_j(t) = \frac{1}{1-e^{-2\pi i \lambda_j c_0}}\sum_{\ell=0}^{\gamma} \binom{\gamma}{\ell}\int_0^{2\pi} \partial_t^\ell\mathcal{H}(t,s) \, \partial_t^{\gamma - \ell} f_j(t-s) ds.	
\end{align*}	
and from  Fa\`a di Bruno formula
$$
\partial_t^\ell  \mathcal{H}(t,s) = \sum_{\Delta(k), \, \ell}
 (-i\lambda_j)^k
 \frac{\ell!}{\ell_1! \cdots \ell_k! }
\left(\prod_{\nu=1}^k
\partial_t^{\ell_\nu-1}(c(t)-c(t-s)) \right)
\mathcal{H}(t,s),
$$	
where
$
\sum\limits_{\Delta(k), \, \ell} = \sum\limits_{k=1}^\ell\sum\limits_{\stackrel{\ell_1+\ldots+\ell_k=\ell}{\ell_\nu \geq 1, \forall \nu}}.
$
Therefore, since
$$\left| \prod_{\nu=1}^k
\partial_t^{\ell_\nu-1}(c(t)-c(t-s))  \right| \leq C^{\ell-k+1}[(\ell -k)!]^\sigma$$
and $\lambda_{j} \leq C_4 j^{m/2n}$ by \eqref{weyl}, we get
\begin{eqnarray}
	|\partial_t^{\gamma} u_j(t)| 
	& \leq & 2 \pi C_2 \sum_{\ell=0}^{\gamma}\left\{ \binom{\gamma}{\ell}
	A_2^{\ell} \sum_{\Delta(k), \, \ell} \frac{\ell!}{\ell_1! \cdots \ell_k! }j^{\frac{km}{2n}} C^{\ell-k+1}[(\ell -k)!]^\sigma \right. \nonumber
	\\ && \times \left.  \sup_{\tau \in [0,2\pi]}
	|\partial_t^{\gamma - \ell} f_j(\tau)| \right\},  \label{sgrunt}
\end{eqnarray}
since $\exp\left(\lambda_j \int_{t-s}^{t} b(r) dr\right) <1$. We obtain an estimate of the same form for 
$b(\cdot) \geq 0$ using \eqref{Solu-2}.

Now, there exist $\epsilon_0 >0$ and $C_3>0$ such that
\begin{equation} \label{sbuff2}
\sup_{\tau \in [0,2\pi]} | \partial_t^{\gamma-\ell} f_j(\tau)| \leq	C_3^{\gamma-\ell+1} [(\gamma-\ell)!]^{\sigma} \exp \left(-\epsilon_0 j^{\frac{1}{2n\mu}} \right).
\end{equation}
Moreover, in view of Lemma \ref{lemma-exp-j} we have
\begin{equation} \label{puff}
	j^{k m/2n} 
	\leq C_{\epsilon_0}^{k} (k !)^{m \mu} \exp \left(\frac{\epsilon_0}{2} j^{\frac{1}{2n\mu}} \right).
\end{equation}


Combining \eqref{sgrunt}, \eqref{sbuff2} and \eqref{puff} we obtain
	$$
|\partial_t^{\gamma} u_j(t)|
\leq C^{\gamma + 1} (\gamma!)^{\max\{\sigma, m\mu\}} \exp\left(-\frac{\epsilon_0}{2} j^{\frac{1}{2n\mu}}
	\right).
$$

Therefore, $u \in  \mathcal{S}_{\max\{\sigma, m\mu\}, \, \mu}$
which implies that $L$ is $\mathcal{S}_{\mu}$-GH.

\end{proof}

\begin{remark}\label{remark-loss-variable}
In contrast with the case when $c(t)$ is constant (see Remark \ref{remark-lost-constant}) we emphasize the possible  
loss of regularity on the variable $t$, namely, for $Lu \in \mathcal{S}_{\sigma, \mu}$ we get $u \in  \mathcal{S}_{\max\{\sigma, m\mu\}, \, \mu}$. This loss depends on $m$, the order of operator $P$, and $\mu$ (the Gelfand-Shilov regularity in $\R^n$).
\end{remark}

In the next result we show that the global hypoellipticity of the operator
$$
L_0 = D_t + c_0P(x,D_x), \ c_0 = (2\pi)^{-1}\int_{0}^{2\pi}c(t)dt.
$$
is a necessary condition for the global hypoellipticity of the operator $L$.

\begin{theorem} \label{LGHimpliesL_0GH}
	If $L$ is $\mathcal{S}_{\mu}$-globally hypoelliptic, then $L_0$ is $\mathcal{S}_{\mu}$-globally hypoelliptic.
\end{theorem}

\begin{proof}	
	We follow the same argument as in the proof of Theorem 3.5 in \cite{AGKM}: we assume by contradiction that $L_0$ is not $\mathcal{S}_{\mu}$-GH and we prove that this implies that $L$ is not $\mathcal{S}_{\mu}$-GH. By Theorem \ref{main_const} it follows that $b_0 = 0$ and  $a_0$ is a $\mu$-exponential Liouville number with respect to the sequence $\{\lambda_j\}$. Since \eqref{non-exp-Liou} implies the condition ii) in Lemma \ref{lt2}, it follows that there exist $\epsilon_0>0$ and a sequence $\{j_{\ell}\}_{\ell \in \N}$ such that $j_{\ell}$ is strictly increasing, $j_{\ell}> \ell$, and 
	\begin{equation}\label{est5}
		|1-e^{-2\pi i a_0\lambda_{j_{\ell}}}| <
		\exp\left (-\epsilon_0 j_{\ell}^{\frac{1}{2n\mu}}\right ) \ \textrm{for all}\  \ell \in\mathbb{N}.
	\end{equation}
	
	By Proposition \ref{Z_finite} we may assume without loss of generality that $c_0\lambda_{j_{\ell}} \notin \Z$ for all $\ell.$ For the sake of simplicity, let us define
	$\mathcal{M}_{j}(t) = \lambda_j c(t)$. We can find a sequence $t_{\ell}\in [0,2\pi]$ such that
	$$
	\int_{0}^{t_{\ell}} \Im\mathcal{M}_{j_{\ell}}(r)dr = \max_{t\in[0,2\pi]} \int_{0}^{t}\Im\mathcal{M}_{j_{\ell}}(r)dr,
	$$
	from which we obtain 
	\begin{equation}\label{negative_IM}
		\int_{t_{\ell}}^{t} \Im \mathcal{M}_{j_{\ell}}(r)dr\leqslant 0, \ \forall 
		t\in[0,2\pi], \ \ell \in \N.
	\end{equation}
	
	Note that by possibly  passing to a subsequence, we may assume that there exists $t_{0}\in[0,2\pi]$ such that $t_{\ell}\rightarrow t_0,$ as $\ell\rightarrow\infty.$ 
	
	Let $I$ be a closed interval in $(0,2\pi)$ such that $t_0\not\in I$ and take $\phi \in\mathcal{G}^{\sigma}_{c}(I,\mathbb{R}),$ such that $0\leqslant \phi(t)\leqslant 1$ and $\int_{0}^{2\pi}\phi(t)dt>0.$ For each $\ell,$ let $f_{\ell}(t)$ be the $2\pi-$periodic extension of
	$$
	(1-e^{-2\pi ia_0 \lambda_{j_{\ell}}}) \exp\left(-\int_{t_{\ell}}^{t} i\mathcal{M}_{j_{\ell}}(r)dr\right)\phi(t), \ t\in[0,2\pi].
	$$

  	A similar approach as in the proof of Theorem \ref{variable-case} shows that the function
	$$
	E_{\ell}(t)=\exp\left(-\int_{t_{\ell}}^{t}i\mathcal{M}_{j_\ell}(r)dr\right)
	$$
	satisfies an estimate of the form
	$$
		|\partial_t^{\gamma} E_{\ell}(t)|  \leq C_{\varepsilon_0}A^{\gamma+1} (\gamma !)^{\max\{\sigma, m\mu \}}\exp \left( \frac{\epsilon_0}{2} j_{\ell}^{\frac{1}{2n\mu}}\right),
	$$
	for $\epsilon_0$ as  in \eqref{est5}. Hence
	\begin{equation*}
		|\partial^\gamma_t f_{\ell}(t)|  \leq 
		C_{\epsilon_0}CA_1^{\gamma} (\gamma!)^{\max\{\sigma, m\mu\}} \exp \left( -\frac{\epsilon_0}{2} j_{\ell}^{\frac{1}{2n\mu}}\right).
	\end{equation*}

    Therefore, 
    $$
    f(t,x)=\sum_{\ell=1}^{\infty} f_{\ell}(t) \varphi_{j_{\ell}}(x) \in \mathcal{S}_{\max\{\sigma, m\mu\}, \, \mu} \subset \mathscr{F}_\mu.
    $$

\smallskip
	
	Now, the  next step is to exhibit an element $u$ of $\mathscr{U}_\mu\setminus \mathscr{F}_\mu$ such that $iLu=f.$ For this, for every $\ell \in \N$, we set
	$$
	u_{\ell}(t) = \frac{1}{1 - e^{-  2 \pi i a_0 \lambda_{j_{\ell}}}} \int_{0}^{2\pi}\exp\left(-i\int_{t-s}^{t}\!\!\mathcal{M}_{j_\ell}(r) \, dr\right) f_{\ell}(t-s)ds,
	$$ 
	which is well defined since $a_0 \lambda_{j_{\ell}} \notin \Z$.
	
	Firstly, note that in case $t,s\in[0,2\pi]$ and  $t-s \geqslant 0$ we have
	\begin{align*}
		|u_{\ell}(t)|	& \leq \left|\frac{1}{ 1-e^{-2\pi ia_0\lambda_{j_{\ell}}}}
		\int_{0}^{2\pi} \exp\left(-\int_{t-s}^{t}i\mathcal{M}_{j_\ell}(r)dr\right){f}_{\ell}(t-s) ds\right| \\
		& \leq\int_{0}^{2\pi} \exp\left(\int_{t-s}^{t} \Im\mathcal{M}_{j_\ell}(r)dr + \int_{t_\ell}^{t-s}\Im\mathcal{M}_{j_\ell}(r)dr\right)ds \\
		& =\int_{0}^{2\pi} \exp\left(\int_{t_{\ell}}^{t}\Im\mathcal{M}_{j_\ell}(r)dr\right)ds \leq 2\pi,
	\end{align*}
 in view of  \eqref{negative_IM}.
	
	On the other hand,  assume	$t,s\in[0,2\pi]$ and   $t-s<0$. Since $f_{\ell}(t-s) = f_{\ell}(t-s+2\pi)$, for all $\ell$, we obtain
	\begin{align*}
		|u_{\ell}(t)|
		& \leq\int_{0}^{2\pi} \exp\left(\int_{t-s}^{t} \Im\mathcal{M}_{j_\ell}(r)dr + \int_{t_\ell}^{t-s+2\pi}\Im\mathcal{M}_{j_\ell}(r)dr\right)ds \\
		& =\int_{0}^{2\pi} \exp\left(\int_{t_{\ell}}^{t} \Im\mathcal{M}_{j_\ell}(r)dr + \int_{t-s}^{t-s+2\pi}\Im\mathcal{M}_{j_\ell}(r)dr\right)ds \\
& =  \int_{0}^{2\pi} \exp\left(\int_{t_{\ell}}^{t} \Im\mathcal{M}_{j_\ell}(r)dr\right)\, ds \leq 2\pi.
	\end{align*}
Therefore,   ${u}_{\ell}(\cdot)$ increases slowly and
	\begin{equation*}
		u = \sum_{\ell \in \N}^{\infty} {u}_{\ell}(t) \varphi_{j_\ell} \in \mathscr{U}_\mu(\T \times \R^n).
	\end{equation*}

	Let $I$ be the interval $[a,b] \subset (0, 2 \pi)$. If $t_0> b,$ then $t_{\ell}> b,$ for all $\ell$ sufficiently large, and
	$$
	|{u}_{\ell}(t_{\ell})|= 
	\int_{t_{\ell}- b}^{t_\ell- a}\phi(t_{\ell}-s)ds= \int_{0}^{2\pi}\phi(t)dt>0.
	$$
	On the other hand, if $t_0<a,$ then $t_{\ell} <a,$ for all $\ell$ sufficiently large, and 
	\begin{align*}
		|{u}_{\ell} (t_{\ell})| & =  
		\left|\int_{t_{\ell}- b +2\pi}^{t_{\ell}- a +2\pi} \exp\left(-\int_{t_{\ell}-s}^{t_{\ell}}i \mathcal{M}_{j_\ell}(r)dr\right)\right. \\
		& \times \left.\exp\left(-\int_{t_\ell}^{t_\ell-s+2\pi}i \mathcal{M}_{\ell}(r)dr\right) \phi(t_\ell -s + 2\pi)ds\right| \\
		& =\int_{0}^{2\pi} \phi(s)ds>0.
	\end{align*}
Hence $|{u}_{\ell} (t_{\ell})| >0$ and it is independent of $\ell$. This implies that $u_{\ell}(\cdot)$ does not satisfy \eqref{deccoeff}, hence
	$u \notin \mathscr{F}_\mu(\T \times \R^n)$. This contraddicts the hypothesis that $L$ is $\mathcal{S}_\mu$-GH.
	
\end{proof}

\begin{proposition}\label{L_0GHimpliesLGH}
	The operator $L = D_t + a(t)P(x,D_x)$ is $\mathcal{S}_{\mu}$-globally hypoelliptic if and only if the same holds true for $L_{a_0} = D_t + a_0P(x,D_x)$, that is, if and only if $a_0$ is not a $\mu$-exponential Liouville number with respect to  $\{\lambda_j\}$. 
\end{proposition}

\begin{proof} By Theorem \ref{LGHimpliesL_0GH} it is sufficient to prove that if $L_{a_0}$ is  $\mathcal{S}_{\mu}$-GH, then $L$ is $\mathcal{S}_{\mu}$-GH.
Now, if $L_{a_0}$ is $\mathcal{S}_{\mu}$-GH then from Theorem \ref{main_const} we have that $a_0$ is not a $\mu$-exponential Liouville number with respect to  $\{\lambda_j\}$. 
Moreover, the set $\mathcal{Z}=\{j \in \N; a_0 \lambda_j \in \Z \}$  is finite and, for any given $f \in \mathcal{F}_\mu$,  the solutions of the equations
$$
\partial_t u_j(t) + ia(t)\lambda_ju_j(t) = f_j(t)
$$
are given by 
\begin{align*}
	u_j(t) = \frac{1}{1 - e^{-2 \pi i \lambda_j a_0}} \int_{0}^{2\pi}\exp\left(- i\lambda_j\int_{t-s}^{t}a(r) \, dr\right) f_j(t-s)ds,
\end{align*}
for $j$ large enough.

By a similar approach as in the proof of Theorem \ref{variable-case} we obtain that there exists $\epsilon_0 >0$ and $C>0$ such that
$$
|\partial_t^{\gamma} u_j(t)|
\leq C^{\gamma + 1}  (\gamma!)^{\max\{\sigma, m\mu\}} \dfrac{1}{|1-e^{-2 \pi i \lambda_j a_0} |}   \exp\left(-\frac{\epsilon_0}{2} j^{\frac{1}{2n\mu}}
\right).
$$
Finally, since $a_0$ is not a $\mu$-exponential Liouville number, 
it follows from Definition \ref{def-non-exp-Liou} and Lemma \ref{lt2}, for $\delta = \epsilon_0/4$ and $\eta = 2n\mu$, that 
$$
|\partial_t^{\gamma} u_j(t)|
\leq C^{\gamma + 1}  (\gamma!)^{\max\{\sigma, m\mu\}}  \exp\left(-\frac{\epsilon_0}{4} j^{\frac{1}{2n\mu}}
\right),
$$
for a new constant $C>0$. Hence, $u \in \mathcal{S}_{ \textrm{max}\{\sigma, m\mu\},\mu}.$
 
\end{proof}

\subsubsection{Change of sign condition}
To conclude the proof of Theorem \ref{main-variable} it remains to prove that if $b \not \equiv 0$ and $b$ changes sign, then $L$ is not $\mathcal{S}_{\mu}$-GH. So, let us investigate the effect of a change of sign condition on  $b(t)$, namely, by admitting the existence of  $t^+$ and $t^{-}$ such that
$$
b(t^+)>0 \ \textrm{ and } \ b(t^-)<0.
$$

To do this we need the following: for each $\eta \in [0,2\pi]$, let  $\mathcal{B}_{\eta}: [0,2\pi] \to \R$ defined by
$$
\mathcal{B}_{\eta}(t) = \int_{\eta}^{t}b(s)ds.
$$

\begin{lemma}\label{lamma-chang.sig}
	Let $b$ be a smooth real $2\pi$-periodic function on $\R$, such that $b \not \equiv 0$ on any interval. Then, the following properties are equivalent:
	\begin{enumerate}
		\item[a)] $b$ changes sign;
		
		\item[b)] there exists $t_0\in \R$ and $t^*, t_* \in ]t_0, t_0+2\pi[$ such that
		\begin{align*}
			\mathcal{B}_{t^*}(t) & \leqslant   0,  \ \forall t \in \   ]t_0,  t_0 + 2\pi], and      \\[2mm]
			\mathcal{B}_{t_*}(t) & \geqslant   0,  \ \forall t \in \ ]t_0,t_0 + 2\pi[;
		\end{align*}
		
		\item[c)] there exists $t_0 \in \R$, partitions
		\begin{align*}
			& t_0 < \alpha^* < \gamma^* < t^* <\delta^* <\beta^* < t_0 +2\pi,  \\[1mm]
			& t_0 < \alpha_* < \gamma_* < t_* <\delta_* <\beta_* < t_0 +2\pi,
		\end{align*}
		and  positive constants $c^*, c_*$ such that the following estimates hold
		\begin{align}
			& \max_{t\in [\alpha^*, \gamma^*] \bigcup [\delta^*,\beta^*] } \mathcal{B}_{t^*} (t)  < -c^*, and  \label{ch-sign-max1b}\\[1mm]
			& \min_{t\in [\alpha_*, \gamma_*] \bigcup [\delta_*,\beta_*] } \mathcal{B}_{t_*} (t)  >  c_*.   \label{ch-sign-min1b}
		\end{align}
	\end{enumerate}
\end{lemma}
\begin{proof}
See   Lemma 5.10 in \cite{AvGrKi}.

\end{proof}

\begin{theorem}\label{change_sign}
Suppose that $b \in \mathcal{G}^{\sigma}(\T)$ is not identically zero on any interval in $[0,2\pi]$. If $b$ changes sign, then $L$ is not $\mathcal{S}_{\mu}$-globally hypoelliptic for any $\mu \geq \frac12$.
\end{theorem}

\begin{proof}The proof is a variant of the proof of Theorem 5.9 in \cite{AvGrKi}.
With the same notation of Lemma \ref{lamma-chang.sig}, set the intervals
\begin{equation*}
	I^* \doteq [\alpha^*, \gamma^*] \cup [\delta^*,\beta^*] \ \ \textrm{ and } \ \
	I_* \doteq [\alpha_*, \gamma_*] \cup [\delta_*,\beta_*],
\end{equation*}
and choose $g^*,g_*, \psi^*,\psi_* \in \mathcal{G}^{\sigma}(\T)$ such that
\begin{align*}
	&  \mbox{supp}(\psi^*) \subset [0,2\pi] \ \mbox{ and } \ \psi^*|_{[\alpha^*, \beta^*]}\equiv 1, \\[2mm]
	&  \mbox{supp}(g^*) \subset [\alpha^*, \beta^*]  \ \mbox{ and } \ g^*|_{[\gamma^*,\delta^*]}\equiv 1,
\end{align*}
and
\begin{align*}
	&  \mbox{supp}(\psi_*) \subset [0,2\pi] \ \mbox{ and } \ \psi_*|_{[\alpha_*, \beta_*]}\equiv 1, \\[2mm]
	&  \mbox{supp}(g_*) \subset [\alpha_*, \beta_*]  \ \mbox{ and } \ g_*|_{[\gamma_*,\delta_*]}\equiv 1.
\end{align*}	
Since $|\lambda_j| \to \infty$, by possibly  passing to a subsequence, we may assume that $\lambda_j>0$, for all $j$, or  $\lambda_j<0$, for all $j$. Let us start with the first case $\lambda_{j}>0$ and define $\{u_j\} \subset \mathcal{G}^{\sigma}(\T)$ by
\begin{equation*}
	u_j(t) = g^*(t) \exp\left[\lambda_j \psi^*(t) (\mathcal{B}_{t^*}(t) - i A_{t^*}(t))\right],
\end{equation*}
where $\mathcal{A}_{\eta}(t) = \int_{\eta}^{t}a(s)ds.$ Then, if $t \in \mbox{supp}(g^*)$ we get
\begin{equation*}
	u_j(t) =
	g^*(t) \exp \left[ \lambda_j(\mathcal{B}_{t^*}(t) - i  A_{t^*}(t))\right],
\end{equation*}
and  $e^{\lambda_j \mathcal{B}_{t^*}(t)}\leqslant 1$, since  $\mathcal{B}_{t^*}(t) \leqslant 0$ on $I^*$.

Therefore, for any $\beta \in \N$ and $t \in \mbox{supp}(g^*)$ we obtain \eqref{estimete-distr}.  Since $|u_{j}(t^*)| = 1$, for any $j$, we have
\begin{equation}\label{u-singular-superlog}
	\{u_{j}(t)\}  \rightsquigarrow u \in  \mathscr{U}_\mu(\T \times \R^n)\setminus \mathscr{F}_\mu(\T \times \R^n),
\end{equation}

Next, consider the sequence 
\begin{equation*}
	f_j(t) = 	-i {g^{*}}'(t) \exp\left[\lambda_j\psi^*(t)(\mathcal{B}_{t^*}(t) - i  A_{t^*}(t))\right].
\end{equation*}

Note that $supp(f_{j}) \subset I^*$, for any $ \in \N$, hence
\begin{align*}
	\left| \partial_t^{k} f_{j}(t) \right| & \leq 
	\sum_{s=0}^{k}\binom{k}{s}\left|\partial_t^{k-s+1}\big(g^*(t)\big)\right| \ \left| \partial_t^{s}\left(\exp \left[ \lambda_j(\mathcal{B}_{t^*}(t) - i  A_{t^*}(t))\right]\right) \right| \\ 
	& \leq C_1^{k+1}\sum_{s=0}^{k}\binom{k}{s} (k!)^{\sigma} j^{km/2n}		\exp(\lambda_j \mathcal{B}_{t^*}(t)) \\
	& \leq C_2^{k+1} (k!)^{\sigma}j^{km/2n}	\exp(\lambda_j \mathcal{B}_{t^*}(t)) \\
	& \leq   C_2^{k+1} (k!)^{\sigma}j^{km/2n} \exp(-c^* \lambda_j) \\ 
	& \leq   C_2^{k+1} (k!)^{\sigma}j^{km/2n} \exp(-\kappa j^{\frac{m}{2n}}),
\end{align*}
for some positive constant $\kappa$, in view of \eqref{weyl}.

Now, by Lemma \ref{lemma-exp-j}, we obtain $C = C(\kappa)$ so that
\begin{align*}
	\left| \partial_t^{k} f_{j}(t) \right| 
	& \leq   C^{k+1} (k!)^{\sigma+1} \exp\left(-\frac{\kappa}{2}j^{\frac{1}{ 2\frac{n}{m}}}\right) \\
	& \leq   C^{k+1} (k!)^{\sigma+1} \exp\left(-\frac{\kappa}{2}j^{\frac{1}{ 2n\mu}}\right),
\end{align*}
where the last inequality is a consequence of the fact that $\mu \geq 1/2$ and $m \geq 2$ imply $\mu\geq 1/m$. Therefore, 
$$
\{f_{j}(t)\}  \rightsquigarrow f \in   \mathscr{F}_{\mu}(\T \times \R^n), \ \mu \geq 1/2,
$$
implying that $L$ is not $\mathcal{S}_{\mu}$-GH, since $Lu=f$.

Finally, we point out that in case $\lambda_j<0$ we can proceed as before by defining the sequences
\begin{equation*}
	u_j(t) = g_*(t) \exp\left[\lambda_j \psi_*(t) (\mathcal{B}_{t_*}(t) - i A_{t_*}(t))\right]
\end{equation*}
and
\begin{equation*}
	f_j(t) = 	-i {g_{*}}'(t) \exp\left[\lambda_j\psi_*(t)(\mathcal{B}_{t_*}(t) - i  A_{t_*}(t))\right].
\end{equation*}

\end{proof}

\begin{remark}
We remark that Theorem \ref{change_sign} can be extended to the following case: there exist an interval $[t_0,t_1] \subset [0,2\pi]$ and $\delta >0$ such that
\begin{align*}
b(t) & > 0, \ \forall t \in(t_0 -\delta, t_0), \\
b(t) & = 0, \  \forall t \in [t_0,t_1], \\
b(t) & < 0, \ \forall t \in(t_1, t_1 + \delta).
\end{align*}

Indeed, in this case we may consider cutoff functions $g_0$ and $g_1$ such that
\begin{align*}
	&  \mbox{supp}(g_0) \subset [t_0-\epsilon,t_0+\epsilon] \ \mbox{ and } \ g_0|_{[t_0-\epsilon/2,t_0+\epsilon/2]}\equiv 1, \\[2mm]
	&  \mbox{supp}(g_1) \subset [t_1-\epsilon,t_1+\epsilon]  \ \mbox{ and } \ g_1|_{[t_1-\epsilon/2,t_1+\epsilon/2]}\equiv 1,
\end{align*}
for $\epsilon$ sufficiently small. Also, we set
\begin{align*}
B_0(t) & = \int_{t_0}^{t} b(s)ds, \ t \in \mbox{supp}(g_0), \\
B_1(t) & = \int_{t_1}^{t} b(s)ds, \ t \in \mbox{supp}(g_1).
\end{align*}

Therefore, the sequence
\begin{equation*}
	u_j(t) =
	g_0(t) \exp \left[ \lambda_j(B_{0}(t) - i  A(t))\right] +
	g_1(t) \exp \left[ \lambda_j(B_{1}(t) - i  A(t))\right],
\end{equation*}
where $A(t) = \int_{0}^{t}a(s)ds$, satisfies
$\{u_{j}(t)\}  \rightsquigarrow u \in  \mathscr{U}_\mu \setminus \mathscr{F}_\mu$ and $Lu \in \mathscr{F}_\mu$.

\end{remark}

\end{document}